\documentclass[titlepage,11pt]{article}

\usepackage[margin=1.2in]{geometry}
\usepackage{latexsym}
\usepackage{amsfonts, amsthm}
\usepackage[leqno]{amsmath}

\usepackage{MnSymbol}%
\usepackage{wasysym}%

\usepackage{thmtools, subcaption}
\usepackage{thm-restate}

\makeatletter
\newcommand{\leqnomode}{\tagsleft@true}
\newcommand{\reqnomode}{\tagsleft@false}
\makeatother

\newenvironment{subproof}[1][\proofname]{%
  \begin{proof}[#1]%
}{%
  \end{proof}%
}

\usepackage{tikz}
\usepackage{float}
\usepackage{url}
\usepackage{parskip}
\usetikzlibrary{arrows}
\usetikzlibrary{decorations.markings}
\usepackage[symbol]{footmisc}

\usepackage{comment}

\usepackage{xspace}
\usepackage[color=Olive]{todonotes}

\usepackage[colorlinks=true,allcolors=blue]{hyperref}
\usepackage{cleveref}

\usepackage[english]{babel}
\addto\extrasenglish{}
\addto\extrasenglish{}
\addto\extrasenglish{}

\usepackage{authblk}

\title{Orientations of cycles in digraphs of high chromatic number and high minimum out-degree}

\author[1]{Hidde Koerts}
\author[2]{Benjamin Moore}
\author[1]{Sophie Spirkl\thanks{Emails: (hkoerts, sspirkl)@uwaterloo.ca, benjamin.moore@ist.ac.at \\
Benjamin Moore is supported by ERC Starting Grant ``RANDSTRUCT'' No.\ 101076777.\\
We acknowledge the support of the Natural Sciences and Engineering Research Council of Canada (NSERC), [funding reference number RGPIN-2020-03912].
Cette recherche a \'et\'e financ\'ee par le Conseil de recherches en sciences naturelles et en g\'enie du Canada (CRSNG), [num\'ero de r\'ef\'erence RGPIN-2020-03912]. This project was funded in part by the Government of Ontario. This research was conducted while Spirkl was an Alfred P. Sloan Fellow.}
}

\affil[1]{University of Waterloo, Department of Combinatorics and Optimization, Waterloo, Canada}

\affil[2]{Institute of Science and Technology Austria, Klosterneuburg, Austria}

\date{\today}

\newtheorem{theorem}{Theorem}[section]

\newtheorem{lemma}[theorem]{Lemma}
\newtheorem{question}[theorem]{Question}
\theoremstyle{remark}
\newtheorem{claim}{Claim}

\theoremstyle{definition}
\newtheorem{definition}[theorem]{Definition}

\begin{document}
\maketitle
\begin{abstract}
    We characterize all orientations of cycles $C$ for which for every fixed $\varepsilon > 0$ there exists a constant $c \geq 1$ such that every digraph $D$ without loops or parallel arcs with $\chi(D) \geq c$ and minimum out-degree at least $\varepsilon |V(D)|$ contains $C$ as a subdigraph. This generalizes a result of Thomassen.
\end{abstract}

\maketitle

\newpage

\section{Introduction}

In 1972, Erdős and Simonovits asked whether every graph with no fixed odd cycle as a subgraph and large minimum degree has bounded chromatic number~\cite{ERDOS1973323}. They specifically considered the case of forbidding a triangle, and additionally posed the question for graphs with no $C_5$ as a subgraph. Their questions can be generalized to the following;

\begin{question}\label{question:cyclesundirected}
    For which $\varepsilon > 0$ and $k > 0$ does there exist a bound $c$ such that each simple graph $G$ either contains $C_{2k+1}$ as a subgraph or $\chi(G) \leq c$ (or both)?
\end{question}
In their paper, Erdős and Simonovits, in collaboration with Hajnal, show that for $C_3$, no such bound exists if $\varepsilon < 1/3$. Thomassen subsequently answered the question in the affirmative for $C_3$ for $\varepsilon > 1/3$~\cite{Thomassen2002}, as well as for all $\varepsilon > 0$ for all longer odd cycles~\cite{Thomassen2007}. An affirmative answer for even cycles follows from the fact that any sufficiently large graph with a linear number of edges contains a large complete bipartite subgraph, as shown by Alon, Krivelevich, and Sudakov~\cite{alon2003turan}. Brandt and Thomass\'{e} additionally showed that all $C_3$-free graphs with minimum out-degree at least $\varepsilon |V(G)|$ for $\varepsilon \geq 1/3$ are in fact four-colourable~\cite{brandt2011dense}.

\autoref{question:cyclesundirected} has been studied extensively in a more general setting under the name \emph{chromatic threshold}, where the fixed odd cycle is replaced by any fixed graph. See for instance~\cite{goddard2011} and~\cite{luczak2010coloring}.
This more general version was completely resolved by Allen, B\"{o}ttcher, Griffiths, Kohayakawa, and Morris~\cite{ALLEN2013261}. 

In this paper, we investigate a directed analogue for the chromatic threshold of cycles. An \emph{orientation}\footnote{We note that the term \emph{orientation} is sometimes used in the literature for digraphs that do not contain loops and for each pair of vertices there exists at most one arc between them. We refer to such digraphs as \emph{simple digraphs}.} of an undirected graph $G$ is a digraph $\vec{G}$ on $V(G)$ where for each edge $uv \in E(G)$ exactly one of the arcs $uv$ and $vu$ is contained in $E(\vec{G})$. The main question we investigate is the following:

\begin{question}\label{question:mainquestion}
    Which orientations of cycles are contained in every digraph with high minimum out-degree and high chromatic number?
\end{question}

Specifically, we will consider digraphs $D$ without loops or parallel arcs with minimum out-degree at least $\varepsilon|V(D)|$ for some fixed $\varepsilon > 0$. Note that we do allow the digraphs to contain anti-parallel arcs, and thus directed $2$-cycles. In this paper, we answer \autoref{question:mainquestion} fully. To formally state our main result, we first introduce some definitions and notation.

We will use sequences of arrows to denote orientations of paths, where the individual arrows correspond to arcs. For instance, $\rightarrow \rightarrow \leftarrow$ corresponds to an orientation of a path on four vertices $v_1$, $v_2$, $v_3$, and $v_4$, given by arcs $v_1v_2$, $v_2v_3$, and $v_4v_3$. If the vertices are specified, we will denote the orientation by $v_1 \rightarrow v_2 \rightarrow v_3 \leftarrow v_4$.

A \emph{block} of an orientation $C$ of a cycle is a maximal connected subdigraph that does not contain $\rightarrow \leftarrow$ or $\leftarrow \rightarrow$ as a subdigraph. Note that if $C$ contains only a single block, then $C$ is a directed cycle, and if $C$ contains at least two blocks, the blocks are exactly the maximal directed subpaths of $C$. Moreover, if $C$ has at least two blocks, the number of blocks is even. Consistent with the usual terminology for paths and cycles, the \emph{length} of a block $B$ is defined as $|E(B)|$.

The main result of this paper is the following theorem, which fully answers \autoref{question:mainquestion}.

\begin{theorem}\label{thm:maintheorem}
    Let $\varepsilon > 0$ and $k \geq 2$ be fixed. For $C$ an orientation of a $k$-cycle, there exists a constant $c \geq 1$ such that every digraph $D$ without loops or parallel arcs with $\chi(D) \geq c$ and minimum out-degree at least $\varepsilon|V(D)|$ contains $C$ as a subdigraph if and only if $C$ either consists of at least three blocks or consists of two blocks, both of length at least two.
\end{theorem}

Observe by reversing all arcs, \autoref{thm:maintheorem} directly implies an analogous result for digraphs of high minimum in-degree. Moreover, by replacing each edge in a graph by two anti-parallel arcs, \autoref{thm:maintheorem} implies the result of Thomassen for odd cycles of length at least five~\cite{Thomassen2007}.

Our proof of the backwards direction of \autoref{thm:maintheorem} is similar in strategy to Thomassen's result for $C_5$ in the undirected case~\cite{Thomassen2007}. His proof proceeds as follows; first he finds a set $S$ of bounded size such that the union of the first and second neighbourhoods of vertices in $S$ contains the entire graph, and second, for each vertex $v \in S$, the union of $v$, its first neighbourhood and its second neighbourhood either contains $C_5$ or is a set of bounded chromatic number. The first part follows in the undirected setting due to the large minimum degree assumption, and we use a similar setup with large minimum out-degree instead. For the second step, Thomassen uses the fact that every graph with large chromatic number contains a long path. We need a digraph analogue of this. In 1980, Burr showed that every digraph of sufficiently large chromatic number contains any fixed-size orientation of a tree~\cite{burr1980subtrees}.

\begin{theorem}[{\cite[Theorem 2]{burr1980subtrees}}]\label{thm:existence}
    For each $k > 0$ there exists a constant $c_k$ such that any orientation of a graph $G$ with $\chi(G) \geq c_k$ contains every oriented tree of order $k$ as a subdigraph.
\end{theorem}

We will use this result extensively for finding subdigraphs of the orientations of cycles. Given \autoref{thm:existence}, it is natural to ask what the optimal values are for the constants $c_k$. Let $c^*_k$ denote this optimal value for $k \geq 0$. That is, any orientation of a graph $G$ with $\chi(G) \geq c^*_k$ contains every oriented tree of order $k$ as a subdigraph, but this does not hold for all orientations of graphs $G$ with $\chi(G) = c^*_k-1$. In his seminal paper, Burr showed that $2k \leq c^*_k \leq (k-1)^2$ for all $k > 0$~\cite{burr1980subtrees}. The best currently known upper bound for $c^*_k$, the first sub-quadratic bound, is due to Bessy, Gon\c{c}alves, and Reinald~\cite{burrbest}.

For the other direction of \autoref{thm:maintheorem}, we provide a construction of digraphs of large minimum out-degree and large chromatic number that do not contain directed cycles or orientations of cycles obtained from a directed cycle by flipping the orientation of a single arc. Where the construction used by Erdős, Hajnal and Simonovits to show that no bound exists for $C_3$ if $\varepsilon < 1/3$ builds on a Kneser graph to provide large chromatic number~\cite{ERDOS1973323}, our construction uses a directed shift graph. The directed shift graph is then augmented to satisfy the large minimum out-degree requirement.

The paper is organized as follows, following the various cases we split the proof of \autoref{thm:maintheorem} into. In \autoref{section:specificsubpaths} we first show that the theorem holds for orientations of cycles containing specific directed subpaths. The results in said section cover most orientations of cycles. There are two types of orientations that are not covered by the results in \autoref{section:specificsubpaths}; directed cycles and orientations obtained from a directed cycle by flipping the orientation of a single arc. \autoref{section:counterexamples} provides constructions for these two types of orientations.

\subsection{Standard notation and terminology}

All digraphs considered in this paper do not contain loops or parallel arcs. The \emph{in-neighbourhood} of a vertex $v\in V(D)$ in a digraph $D$, denoted by $N^-_D(v)$, is the set of all vertices $u \in V(D)$ such that $uv \in E(D)$. Similarly, the \emph{out-neighbourhood}, denoted by $N^+_D(v)$, is the set of all vertices $u \in V(D)$ such that $vu \in E(D)$. We refer to the vertices in $N^-_D(v)$ and $N^+_D(v)$ as the \emph{in-neighbours} and \emph{out-neighbours} of $v$ respectively. Moreover, the \emph{closed in-neighbourhood} and \emph{closed out-neighbourhood} of a vertex $v \in V(D)$, denoted by $N^-_D[v]$ and $N^+_D[v]$, are defined as $N^-_D(v) \cup \{v\}$ and $N^+_D(v) \cup \{v\}$ respectively. The \emph{in-degree} of a vertex $v\in V(D)$ in a digraph $D$, denoted by $\deg^-_D(v)$, is then defined as $|N^-_D(v)|$. Analogously, the \emph{out-degree}, denoted by $\deg^+_D(v)$, is given by $|N^+_D(v)|$. If the digraph is clear from the context, we omit the subscript from the notation of the neighbourhoods and degrees. For a (di)graph $G$ and a vertex set $S \subseteq V(G)$, we use $G[S]$ to denote the sub(di)graph induced by $S$. For two disjoint sets $A, B \subseteq V(D)$, we say that $A$ is \emph{out-complete} to $B$ (or equivalently $B$ is \emph{in-complete} from $A$) if for all $a \in A$ and $b \in B$ we have that $ab \in E(D)$. We will use similar terminology if one of the two sets is a single vertex. Finally, for an integer $k > 0$ we use $[k]$ to denote the set $\{1, \ldots, k\}$.

\section{Orientations of cycles containing $\rightarrow\leftarrow\rightarrow\leftarrow$ or $\rightarrow\rightarrow\leftarrow\leftarrow$}\label{section:specificsubpaths}

In this section, we show the main result for orientations of cycles containing a subdigraph of the form $\rightarrow\leftarrow\rightarrow\leftarrow$ or of the form $\rightarrow\rightarrow\leftarrow\leftarrow$. We will subsequently use the results for these specific cases to reduce the number of cases we have to consider in the remainder of the paper significantly.

\begin{lemma}\label{lem:rlrl}
    Let $\varepsilon > 0$ and an integer $k \geq 4$ be fixed, and either let $C$ be an orientation of a $k$-cycle containing $\rightarrow \leftarrow \rightarrow \leftarrow$ as a subdigraph if $k \geq 5$, or let $C$ be the orientation of a $4$-cycle obtained by identifying the endpoints of $\rightarrow\leftarrow\rightarrow\leftarrow$ if $k = 4$. Let $D$ be a digraph with $|V(D)| \geq 12/\varepsilon^2$, $\chi(D) \geq 4c^*_{k-1}/\varepsilon$ and minimum out-degree at least $\varepsilon |V(D)|$. Then $D$ contains a subdigraph isomorphic to $C$. 
\end{lemma}
\begin{proof}
    First, let $u_1, u_2, u_3, u_4, u_5 \in V(C)$ be consecutive vertices along $C$ such that $u_1 \rightarrow u_2 \leftarrow u_3 \rightarrow u_4 \leftarrow u_5$, where $u_1 = u_5$ if $k = 4$. We aim to first find $C - u_3$ or $C - \{u_2, u_3, u_4\}$ as a subdigraph in $D$, and subsequently extend said subdigraph to a subdigraph isomorphic to $C$. 

    Let $v_1, \ldots, v_{\ell}$ be a maximal sequence of vertices such that for all $i \in [\ell]$ it holds that
    \[\left| N^+(v_i) \cap \bigcup_{j=1}^{i-1} N^+[v_j]\right| < \frac{|N^+(v_i)|}{2}.\] Note that such a sequence exists as any sequence consisting of a single vertex satisfies the condition. Moreover, for $i \in [\ell]$ let 
    \[S_i := N^+(v_i) \setminus \bigcup_{j=1}^{i-1} N^+[v_j].\]
    Suppose that there exists an index $i \in [\ell]$ such that $\chi(D[S_i]) \geq c^*_{k-1}$. Then, by the definition of $c^*_{k-1}$, the digraph $D[S_i]$ contains a subdigraph $P$ isomorphic to $C - u_3$. Then, for $e_1, e_2$ the arcs from $v_i$ to the endpoints of $P$, we have that $P + e_1 + e_2$  is a subdigraph of $D$ isomorphic to $C$, as desired. Hence, we may assume that $\chi(D[S_i]) < c^*_{k-1}$ for all $i \in [\ell]$.
    
    By the definitions of the sequence $v_1, \ldots, v_{\ell}$ and the sets $S_1, \ldots, S_{\ell}$, and using the minimum out-degree of the digraph, we observe that
    \[|S_i| \geq \frac{|N^+(v_i)|}{2} \geq \frac{\varepsilon |V(D)|}{2}\] for all $i \in [\ell]$. Note that as sets $S_1, \ldots, S_{\ell}$ are disjoint, it follows that $\ell \leq |V(D)|/\left( \frac{\varepsilon|V(D)|}{2}\right) = 2/\varepsilon$.     

    Let $X:= V(G) \setminus \left( \bigcup_{i=1}^{\ell} S_i \cup \{v_i\}\right)$. By the sub-additivity of the chromatic number, and as $\ell \leq 2/\varepsilon$, it follows that 
    \[\chi(D[X]) \geq \chi(D) - \sum_{i = 1}^{\ell} \chi(D[S_i \cup \{v_i\}]) \geq 4c^*_{k-1}/\varepsilon - \sum_{i=1}^{\ell} c^*_{k-1} \geq 2c^*_{k-1}/\varepsilon.\]    
    Moreover, by the maximality of $\ell$, each vertex $v \in X$ has strictly fewer than $|N^+(v)|/2$ out-neighbours in $X$. Therefore, each such vertex $v \in X$ has at least $|N^+(v)|/2 \geq \varepsilon |V(D)|/2$ out-neighbours in $\bigcup_{i=1}^{\ell} (S_i \cup \{v_i\})$ and therefore at least $\varepsilon |V(D)|/2 - \ell$ out-neighbours in $\bigcup_{i=1}^{\ell} S_i$. 

    Let $X_1, \ldots, X_{\ell}$ be a partition of $X$ where $X_i$ for $i \in [\ell]$ is the set of vertices $v \in X$ where $|N^+(v) \cap S_j|$ is maximized for $j = i$ (in the case that this value is maximized for multiple indices, we assign the vertex arbitrarily to one of the corresponding sets). For each $i \in [\ell]$ and every vertex $v \in X_i$, it then holds that $|N^+(v) \cap S_i| \geq (\varepsilon |V(D)|/2 - \ell)/\ell \geq \varepsilon^2 |V(D)|/4 - 1$. Because $|V(D)| \geq 12/\varepsilon^2$, we then obtain that $|N^+(v) \cap S_i| \geq 2$.
    
    Furthermore, there exists an index $i^* \in [\ell]$ such that $\chi(D[X_{i^*}]) \geq \chi(D[X])/ \ell$. Then, by the previously established bound on $\chi(D[X])$,
    \[\chi(D[X_{i^*}]) \geq \frac{2c^*_{k-1}/\varepsilon}{2 / \varepsilon} = c^*_{k-1} \geq c^*_{k-3}\]
    By the definition of $c^*_{k-3}$, the digraph $D[X_{i^*}]$ contains a subdigraph $P$ isomorphic to $C-\{u_2, u_3, u_4\}$. Let $x, y \in X_{i^*}$ be the two endpoints of the path $P$. Note that if $k =4$, these endpoints are not distinct. As $x, y \in X_{i^*}$, we have that $|N^+(x) \cap S_{i^*}| \geq 2$ and $|N^+(y) \cap S_{i^*}| \geq 2$. Let $x', y'$ be distinct vertices in $S_{i^*}$ such that $xx', yy' \in E(D)$. Then $P+xx'+yy' +v_{i^*}x' + v_{i^*}y'$ is a subdigraph of $D$ isomorphic to $C$, as desired.    
\end{proof}

The structure of the proof of the next lemma is similar to that used in the proof of \autoref{lem:rlrl}.

\begin{lemma}\label{lem:rrll}
    Let $\varepsilon > 0$ and an integer $k \geq 4$ be fixed, and either let $C$ be an orientation of a $k$-cycle containing $\rightarrow \rightarrow \leftarrow \leftarrow$ as a subdigraph if $k \geq 5$, or let $C$ be the orientation of a $4$-cycle obtained by identifying the endpoints of $\rightarrow\rightarrow\leftarrow\leftarrow$ if $k = 4$. Let $D$ be a digraph with $|V(D)| \geq 48/\varepsilon^3 $, $\chi(D) \geq 16c^*_{k-1}/\varepsilon^2$ and minimum out-degree at least $\varepsilon |V(D)|$. Then $D$ contains a subdigraph isomorphic to $C$. 
\end{lemma}
\begin{proof}
    First, let $u_1, u_2, u_3, u_4, u_5 \in V(C)$ be consecutive vertices along $C$ such that $u_1 \rightarrow u_2 \rightarrow u_3 \leftarrow u_4 \leftarrow u_5$, where $u_1 = u_5$ if $k = 4$. We aim to first find $C - \{u_2, u_3, u_4\}$ as a subdigraph in $D$, and subsequently extend said subdigraph to a subdigraph isomorphic to $C$. 

    Let $v_1, \ldots, v_{\ell}$ be a maximal sequence of vertices such that for all $i \in [\ell]$ it holds that
    \[\left| N^-(v_i) \setminus \bigcup_{j=1}^{i-1} N^-[v_j]\right| \geq \frac{\varepsilon^2 |V(D)|}{8}.\] Note that there at least exist one vertex of in-degree at least $\varepsilon |V(D)|$ by an averaging argument using the minimum degree condition on $D$. Since any sequence consisting of a single vertex satisfies the condition, the aforementioned maximal sequence exists. Moreover, for $i \in [\ell]$ let 
    \[S_i := N^-(v_i) \setminus \bigcup_{j=1}^{i-1} N^-[v_j].\]

    Suppose that there exists an index $i\in [\ell]$ such that $\chi(D[S_i]) \geq c^*_{k-1}$. Then, by the definition of $c^*_{k-1}$, the digraph $D[S_i]$ contains a subdigraph $P$ isomorphic to $C - u_3$. Then, for $x, y \in S_i$ the endpoints of path $P$, the path $P+xv_i + yv_i$  is a subdigraph of $D$ isomorphic to $C$, as desired. Hence, we may assume that $\chi(D[S_i]) < c^*_{k-1}$ for all $i \in [\ell]$.
    
    Next, by the definition of $v_i$, we have that $|S_i| \geq \varepsilon^2 |V(D)|/8$ for all $i \in [\ell]$. Note that as the sets $S_1, \ldots, S_{\ell}$ are disjoint, it follows that $\ell \leq \frac{|V(D)|}{\varepsilon^2 |V(D)|/8} = 8/\varepsilon^2$.
    
    Let $X:= V(G) \setminus \left( \bigcup_{i=1}^{\ell} S_i \cup \{v_i\}\right)$. By the sub-additivity of the chromatic number, and as $\ell \leq 8/\varepsilon^2$, it follows that 
    \[\chi(D[X]) \geq \chi(D) - \sum_{i = 1}^{\ell} \chi(D[S_i \cup \{v_i\}]) \geq 16c^*_{k-1}/\varepsilon^2 - \ell \cdot c^*_{k-1} \geq 8/\varepsilon^2 \cdot c^*_{k-3}.\] Moreover, by the maximality of $\ell$, each vertex $v \in X$ has strictly fewer than $\varepsilon^2 |V(D)|/8$ in-neighbours in $X$. 

    Let $X_A \subseteq X$ be the set of all vertices $v \in X$ such that $|N^+(v) \cap \bigcup_{i=1}^{\ell}(S_i \cup \{v_i\})| \geq \varepsilon |V(D)| / 2$, and let $X_B := X \setminus X_A$. Moreover, let $X_{A, 1}, \ldots, X_{A, \ell}$ be a partition of $X_A$ where $X_{A,i}$ for $i \in [\ell]$ is the set of vertices $u \in X_A$ for which $|N^+(u) \cap S_j|$ is maximized for $j = i$ (in the case that this value is maximized for multiple indices, we again assign the vertex arbitrarily to one of the corresponding sets). We obtain that, for every $v \in X_{A,i}$ for every $i \in [\ell]$,
    \[|N^+(v) \cap S_i| \geq \frac{\varepsilon |V(D)|/2 - \ell}{\ell} \geq \varepsilon^3 |V(D)| / 16 - 1.\]
    Because $|V(D)| \geq 48/\varepsilon^3$, this gives that $|N^+(v) \cap S_i| \geq 2$.

    Now suppose that there exists an index $i \in [\ell]$ such that $\chi(D[X_{A,i}]) \geq c^*_{k-3}$. Then, by the definition of $c^*_{k-3}$, the digraph $D$ contains a subdigraph $P$ isomorphic to $C - \{u_2, u_3, u_4\}$. Let $x, y \in X_{A,i}$ be the two endpoints of the path $P$. Note that if $k =4$, these endpoints are not distinct. Because $x, y \in X_{A, i}$, and thus $|N^+(x) \cap (S_i)|\geq 2$, there exist distinct out-neighbours $x'$ and $y'$ of $x$ and $y$, respectively, in $S_i$. Then $P + xx' + yy' + x'v_i + y'v_i$ is a subdigraph of $D$ isomorphic to $C$.

    Hence, we may assume that $\chi(D[X_{A,i}]) < c^*_{k-3} \leq c^*_{k-1}$ for all $i \in [\ell]$. Therefore,
    \[\chi(D[X_B]) \geq \chi(D[X]) - \sum_{i=1}^{\ell} \chi(D[X_{A,i}]) \geq 8/\varepsilon^2 \cdot c^*_{k-1} - \ell (c^*_{k-1} - 1) \geq 8/\varepsilon \geq 8.\]
    Thus, $X_B \neq \emptyset$. Note that for each vertex $v \in X_B$, since $v \not \in X_A$, it holds that $|N^+(v) \cap \bigcup_{i=1}^{\ell} (S_i \cup \{v_i\})|< \varepsilon |V(D)|/2$. Then, as each vertex has out-degree at least $\varepsilon |V(D)|$, it holds that $|N^+(v) \cap X| > \varepsilon |V(D)|/2$ for all $v \in X_B$.

    First suppose that $|N^+(v) \cap X_B| \geq \varepsilon |V(D)|/4$ for every vertex $v \in X_B$. Then there exists a vertex $v^* \in X_B$ such that $|N^-(v^*) \cap X_B| \geq \varepsilon |V(D)|/4$. However, since $X_B$ is disjoint from $S_1, \ldots, S_{\ell}$ and $\{v_1, \ldots, v_{\ell}\}$, and as $\varepsilon |V(D)|/4 \geq \varepsilon^2 |V(D)|/8$, the existence of such a vertex contradicts the maximality of $\ell$.

    Thus, there exists a vertex $v' \in X_B$ such that $|N^+(v') \cap X_B| < \varepsilon |V(D)|/4$. Then, as additionally $|N^+(v') \cap X| > \varepsilon |V(D)|/2$ as $v' \in X_B$, it follows that $|N^+(v') \cap X_A| > \varepsilon |V(D)|/4$ and therefore, $|X_A| > \varepsilon |V(D)|/4$. By the definition of $X_A$, there then exist more than $\varepsilon |V(D)|/4 \cdot \varepsilon |V(D)|/2 = \varepsilon^2 |V(D)|^2/8$ arcs from $X_A$ to $\bigcup_{i \in [\ell]} (S_i \cup \{v_i\})$. Moreover, since each in-neighbour of $v_i$ for $i \in [\ell]$ is contained in $\bigcup_{j = 1}^i (S_j \cup \{v_j\})$, in fact there exist more than $\varepsilon^2 |V(D)|^2/8$ arcs from $X_A$ to $\left(\bigcup_{i \in [\ell]} S_i\right) \setminus \{v_i \, : \, i \in [\ell]\}$. But then, as $\left| \bigcup_{i \in [\ell]} S_i \right| \leq |V(D)|$, there exists a vertex $v'' \in \left(\bigcup_{i \in [\ell]} S_i\right) \setminus \{v_i \, : \, i \in [\ell]\}$ such that $v''$ has at least $\varepsilon^2 |V(D)|/8$ in-neighbours in $X_A$. Setting $v_{\ell +1} := v''$ then contradicts the maximality of $\ell$. By this contradiction, we conclude that the result holds.
\end{proof}

Based on \autoref{lem:rlrl} and \autoref{lem:rrll}, we will now show that \autoref{thm:maintheorem} holds for orientations of cycles consisting of at least three blocks, and orientations of cycles consisting of two blocks, both of length at least two. We first introduce a useful definition we will use in the upcoming proofs. 

For a set $S \subseteq V(D)$ and an integer $r > 0$, let $N_r^-(S) \subseteq V(D)$ denote the set of vertices outside of $S$ that have at least $r$ out-neighbours in $S$. That is, $N_r^-(S) := \{v \in V(D) \setminus S \, : \, |N^+(v) \cap S| \geq r\}$. We then have the following definition:

\begin{definition}
    Let $D$ be a digraph and let $0 < c < 1$ be fixed. A set $X \subseteq V(D)$ is \emph{$(c, r)$-cohesive} if for all $v \in X$ it holds that $\chi(D[X \setminus  N_r^-(N^+(v))]) \leq c\cdot \chi(D[X])$.
\end{definition}

This definition allows us to quantify the intuition that there exist sets of vertices where most vertices have many common out-neighbours. The following lemma establishes that there exist $(c,r)$-cohesive sets in the digraphs we consider.

\begin{lemma}\label{lem:cohesive}
    Let $\varepsilon > 0$, $0 < c < 1$, and integers $m > 0$ and $r > 0$ be fixed. Let $\ell := \lceil 2/\varepsilon\rceil$ and let $D$ be a digraph with $|V(D)| \geq \ell^2\cdot(r-1)$, $\chi(D) \geq m/c^{\ell - 1}$ and minimum out-degree at least $\varepsilon |V(D)|$. Then $D$ contains a $(c,r)$-cohesive set $X\subseteq V(D)$ such that $\chi(D[X]) \geq m$.
\end{lemma}
\begin{proof}
	Suppose that $D$ does not contain such a vertex set. We define a sequence $X_1, \ldots, X_{\ell} \subseteq V(D)$ of vertex subsets and a corresponding sequence $v_1, \ldots, v_{\ell} \in V(D)$ of vertices as follows; let $X_1 := V(D)$, and let $v_1 \in X_1$ be a vertex witnessing that $X_1$ is not $(c,r)$-cohesive. Then, for $i \in [\ell]\setminus\{1\}$, iteratively define $X_i := X_{i-1}\setminus  N_r^-(N^+(v_{i-1}))$ and $v_i \in X_i$ a vertex that witnesses that $X_i$ is not $(c,r)$-cohesive. That is, for all $i \in [\ell]$, it holds that $\chi(D[X_i \setminus N_r^-(N^+(v_i))]) > c \cdot \chi(D[X_i])$.
	
	Moreover, by the definition of the sets $X_i$, it then follows for all $i \in [\ell]\setminus\{1\}$ that $\chi(D[X_i]) > c \cdot \chi(D[X_{i-1}])$. Thus, for all $i \in [\ell]$, we have that $\chi(D[X_i]) > c^{i-1} \chi(D[X_1]) = c^{i-1} \chi(D) \geq c^{\ell - 1} \chi(D) = m$. Hence, by our assumption that $D$ does not contain a $(c,r)$-cohesive set $X \subseteq V(D)$ with $\chi(D[X]) \geq m$, the sequences $X_1, \ldots , X_{\ell}$ and $v_1, \ldots , v_{\ell}$ are well-defined.
    
	We then observe that since $X_i\subsetneq X_{i-1}$ for all $i \in [\ell]\setminus\{1\}$, and as $X_{i} = X_{i-1} \setminus N_r^-(N^+(v_{i-1}))$, it holds that $X_j \subseteq X_{i} \setminus N_r^-(N^+(v_{i}))$ for all $1 \leq i < j \leq \ell$. Therefore, $v_j \not \in N_r^-(N^+(v_i))$ for all $1 \leq i < j \leq \ell$, and hence $|N^+(v_i) \cap N^+(v_j)| \leq r-1$. Since $D$ has minimum out-degree at least $\varepsilon |V(D)|$, it follows that 
    \[|N^+(\{v_1, \ldots, v_{\ell}\})| > \ell \cdot \varepsilon |V(D)| - \ell^2\cdot (r-1) \geq \lceil 2/\varepsilon \rceil \cdot \varepsilon |V(D)|- \ell^2\cdot (r-1) \geq 2|V(D)| - \ell^2\cdot (r-1).\] 
    Then, as $|V(D)| \geq \ell^2\cdot (r-1)$, we obtain that $|N^+(\{v_1, \ldots, v_{\ell}\})| > |V(D)|$, a contradiction. Hence, we conclude that $D$ must contain a $(c,r)$-cohesive set $X\subseteq V(D)$ with $\chi(D[X]) \geq m$, as desired.
\end{proof}

We now use this lemma to lift \autoref{lem:rlrl} and \autoref{lem:rrll} to a general result for all orientations of $k$-cycles containing at least three blocks.

\begin{lemma}\label{lem:threeblocks}
    Let $0 < \varepsilon < 1/2$ and an integer $k \geq 4$ be fixed. Let $D$ be a digraph with $|V(D)| \geq 48k/\varepsilon^3$, $\chi(D) \geq 16c_{k-1}^*/(\varepsilon^{\lceil 2/\varepsilon \rceil} (1-2\varepsilon))$ and minimum out-degree at least $\varepsilon|V(D)|$, and let $C$ be an orientation of a $k$-cycle consisting of more than two blocks. Then $D$ contains a subdigraph isomorphic to $C$.
\end{lemma}
\begin{proof}
    Let $V(C) = \{u_1, \ldots , u_k\}$ such that $u_1, \ldots, u_k$ is a cyclic ordering of the underlying undirected graph. By \autoref{lem:rlrl} and \autoref{lem:rrll} we may assume that $k \geq 5$ and $C$ does not contain $\rightarrow\leftarrow\rightarrow\leftarrow$ or $\rightarrow\rightarrow\leftarrow\leftarrow$ as a subdigraph respectively.

    Since $C$ consists of more than two blocks, it consists of at least four blocks by parity. Let $B_1, \ldots, B_r$ with $r \geq 4$ and $r$ even be the blocks of $C$ ordered as they appear along the cycle according to the ordering $u_1, \ldots, u_k$. For the sake of convenience, relabel the vertices $u_1, \ldots, u_k$ such that $B_1$ is given by the directed path $u_1 \rightarrow u_2 \rightarrow \ldots \rightarrow u_{\ell}$ where $\ell - 1$ is the length of block $B_1$. As $C$ does not contain $\rightarrow\leftarrow\rightarrow\leftarrow$ as a subdigraph and as $k \geq 5$, at least one of the blocks of $C$ has length greater than one. Without loss of generality, suppose that $B_1$ has length at least two. Then, as $C$ does not contain $\rightarrow\rightarrow\leftarrow\leftarrow$ as a subdigraph, it follows that $B_2$ has length exactly one. Similarly, since $C$ does not contain $\rightarrow\leftarrow\rightarrow\leftarrow$, block $B_3$ has length at least two. It follows that all blocks $B_i$ with $i \in [r]$ odd have length at least $2$, and all blocks $B_i$ with $i \in [r]$ even have length exactly $1$. 

    By \autoref{lem:cohesive}, applied with $c = \varepsilon$, $r = k+1$, and $m = 16c_{k-1}^*/(1-2\varepsilon)$, digraph $D$ contains an $(\varepsilon, k+1)$-cohesive set $S$ with $\chi(D[S]) \geq 16c^*_{k-1} / (1-2\varepsilon)$. Since $\chi(D[S]) \geq c^*_{k-1}$, there exists a directed path $P = v_1 \rightarrow \ldots \rightarrow v_{\ell-1}$ in $D[S]$. We will use this path as the first building block of a copy of $C$, corresponding to the first $\ell -1$ vertices in $B_1$. To attach the rest of the desired cycle, consider the vertices $v_1$ and $v_{\ell - 1}$. Let $X_1 := N^+_D(v_1) \setminus V(P)$ and let $X_{\ell - 1} := N^+_D(v_{\ell-1}) \setminus V(P)$. As $|V(D)| > k/\varepsilon$, we have that $|V(P)| < \ell < k < \varepsilon |V(D)| \leq \deg^+_D(v_j)$ for $j \in \{1, \ell-1\}$, and hence $X_1$ and $X_{\ell - 1}$ are non-empty.

    Next, let $Y_1 := S \setminus N_{k+1}^-(N^+(v_1))$ and $Y_{\ell-1} := S \setminus N_{k+1}^-(N^+(v_{\ell -1}))$, and let $S' = S\setminus (Y_1 \cup Y_{\ell -1} \cup V(P))$. We then observe that as $S$ is $(\varepsilon, k+1)$-cohesive, it follows that 
    \[\chi(D[S']) \geq (1 - 2\varepsilon)\chi(D[S]) - |V(P)| > 16c^*_{k-1} - k.\]
    Since $c^*_k \geq 2k$ as shown by Burr~\cite{burr1980subtrees}, we obtain that $\chi(D[S']) > c^*_{k-1}$. Moreover, $S' \subseteq N_{k+1}^-(N^+(v_1)) \cap N_{k+1}^-(N^+(v_{\ell -1}))$, and thus each vertex in $S'$ has at least $k+1$ out-neighbours in both $N^+(v_1)$ and $N^+(v_{\ell})$. Then, since $|V(P)| < k$, each vertex in $S'$ has at least two out-neighbours in both $X_1$ and $X_{\ell-1}$.

    In $D[S']$, as $\chi(D[S']) > c^*_{k-1} \geq c^*_{k - \ell - 1}$, there exists a copy of the digraph $C - (\{u_1, \ldots, u_{\ell}\} \cup \{u_k\})$; let $Q$ be this subdigraph. Note that the underlying undirected graph is a path. Let $y$ be the endpoint of said path corresponding to vertex $u_{\ell + 1}$ and let $z$ be the endpoint corresponding to $u_{k-1}$.  

    Note that as $y, z \in V(Q) \subseteq S'$, both have at least two out-neighbours in both $X_1$ and $X_{\ell -1}$. Let $y'$ be an out-neighbour of $y$ in $X_{\ell -1}$, and let $z'$ be an out-neighbour of $z$ in $X_1 \setminus \{y'\}$. Then $P$, $y'$, $Q$, and $z'$ form a copy of $C$ in digraph $D$, where $P \cup \{y'\}$ corresponds to $B_1$, the blocks $B_2, \ldots, B_{r-1}$ are all contained in $Q + yy' + zz'$, and $B_r$ is given by $z' \leftarrow v_1$. Observe here that $B_r$ is a block of length one due to parity.
\end{proof}

Note that for the application of \autoref{lem:threeblocks} towards proving \autoref{thm:maintheorem}, the restriction that $\varepsilon < 1/2$ is not of significance, since smaller values of $\varepsilon$ lead to strictly stronger statements.

Moreover, we can use the previous result on orientations of cycles containing $\rightarrow\rightarrow\leftarrow\leftarrow$ to show the following;

\begin{lemma}\label{lem:twoblockseachlengthtwo}
    Let $\varepsilon > 0$ and an integer $k \geq 4$ be fixed. Let $D$ be a digraph with $|V(D)| \geq 48/\varepsilon^3$, $\chi(D) \geq 16c_{k-1}^* / \varepsilon^2$ and minimum out-degree at least $\varepsilon |V(D)|$, and let $C$ be an orientation of a $k$-cycle consisting of two blocks, each of length at least $2$. Then $D$ contains a subdigraph isomorphic to $C$.
\end{lemma}
\begin{proof}
    The result follows directly from \autoref{lem:rrll} and the fact that each block has length at least $2$.
\end{proof}

\section{Constructions for directed cycles and orientations of cycles with a single arc flipped}\label{section:counterexamples}
In this section, we provide constructions of digraphs with arbitrarily large chromatic number and minimum out-degree that do not contain a specified orientation of a fixed-length cycle as a subdigraph. 
\begin{theorem}\label{thm:counterexample-directed}
    For each integer $k \geq 2$ there exists an $\varepsilon > 0$ such that for all integers $n$ and $c$, there exists a digraph $D$ with $|V(D)| \geq n$, $\chi(D) \geq c$, and minimum out-degree at least $\varepsilon |V(D)|$, such that $D$ does not contain a directed $k$-cycle as a subdigraph.
\end{theorem}
\begin{proof}
    Let $\varepsilon < 1/(k+1)$, and let $\ell = \max(n/(k+1), c)$. Consider the digraph $D$ obtained from a directed $(k+1)$-cycle by blowing up each vertex into a transitive tournament on $\ell$ vertices. Hence, digraph $D$ satisfies $|V(D)| \geq (k+1)\cdot \ell \geq n$, $\chi(D) \geq \ell \geq c$, and minimum out-degree at least $\ell = |V(D)|/(k+1) > \varepsilon |V(D)|$. Moreover, we observe that this digraph does not contain a directed $k$-cycle as a subdigraph. Thus, digraph $D$ satisfies all the desired conditions. 
\end{proof}

Next, we similarly provide a construction for cycles containing exactly two blocks, one of which has length one. Observe that such cycles can be obtained from a directed cycle by flipping the orientation of a single arc. We first prove an auxiliary lemma we will need for the construction. We say a digraph $D'$ is obtained from a digraph $D$ by \emph{cloning a vertex $v \in V(D)$} if $D'$ may be constructed from $D$ by adding a new vertex $v'$ with adjacency $N^-(v') := N^-(v)$ and $N^+(v'):= N^+(v)$. We say that $v'$ is a \emph{clone} of $v$.

For $k \geq 3$, let $\mathcal{C}_k$ be the family of all directed cycles of length at most $k$, and let $\mathcal{C}_k'$ be the family of all orientations of cycles of length at most $k$ containing exactly two blocks, of which one has length one.

\begin{lemma}\label{lem:cloning}
     For $k \geq 3$ fixed, let $D$ be a digraph that does not contain a digraph in $\mathcal{C}_k \cup \mathcal{C}_k'$ as a subdigraph. Then the digraph $D'$ resulting from cloning a vertex $v \in V(D)$ does not contain any digraph in $\mathcal{C}_k \cup \mathcal{C}_k'$ as a subdigraph either.
\end{lemma}
\begin{proof}
    Let $D'$ be the digraph resulting from cloning a vertex $v \in V(D)$. Let $v'$ be the clone of $v$ in $D'$. Suppose that $D'$ does contain a subdigraph $H$ isomorphic to a digraph $C \in \mathcal{C}_k \cup \mathcal{C}_k'$. If $H$ does not contain both $v$ and $v'$, we may assume without loss of generality that $v \in V(H)$. But then $H$ is also a subdigraph of $D$, a contradiction.  Hence we may assume that $v, v' \in V(H)$. 

    Let $P_1$ and $P_2$ be the two subdigraphs of $H$ such that $P_1$ and $P_2$ are both orientations of paths, $H = P_1 \cup P_2$, and $V(P_1) \cap V(P_2) = \{v, v'\}$. We claim that at least one of $P_1$ and $P_2$ is a directed path. If $H$ is a directed cycle, both $P_1$ and $P_2$ are directed paths. Otherwise, if $H$ is isomorphic to a digraph in $\mathcal{C}_k'$ and thus contains two blocks, one of length one, either $P_1$ or $P_2$ is a subdigraph of one of the blocks of $H$. Hence, either $P_1$ or $P_2$ is a directed path. Without loss of generality, assume that $P_1$ is a directed path. Because $v$ and $v'$ are clones and thus are non-adjacent, $P_1$ has length at least two. But then, we may identify endpoints $v$ and $v'$ in $P_1$ to obtain a subdigraph $P_1'$ of $D$. Moreover, $P_1'$ is a directed cycle of length at most $|V(H)| - 1$. Thus, as $|V(H)| \leq k$, we find that $P_1'$ is a subdigraph of $D$ isomorphic to a digraph in $\mathcal{C}_k$, a contradiction.
\end{proof}

We now proceed to the construction for cycles containing exactly two blocks, of which one of length one. The construction is based on a large directed shift graph, with further modifications ensuring the desired minimum degree. 

Let $\vec{S}_{m,r}$ denote the directed shift graph with tuples of size $r$ over an alphabet of size $m$. That is, $V(\vec{S}_{m,r}) = \{(a_1, \ldots, a_r) \in [m]^r \, : \, a_1 < \ldots < a_r\}$ and a vertex $(a_1, \ldots, a_r) \in V(\vec{S}_{m,r})$ is out-adjacent to $(b_1, \ldots, b_r) \in V(\vec{S}_{m,r})$ if $a_{i+1} = b_i$ for all $i \in [r-1]$. Note that the directed shift graph is the natural orientation of a shift graph. See \autoref{fig:shift} for an example. As shown by Erdős and Hajnal~\cite{erdos1968chromatic}, $\chi(\vec{S}_{m,r})$ tends to infinity as $m \to \infty$ for all $r \geq 2$.

\begin{figure}
    \centering
    \includegraphics[width=0.4\linewidth]{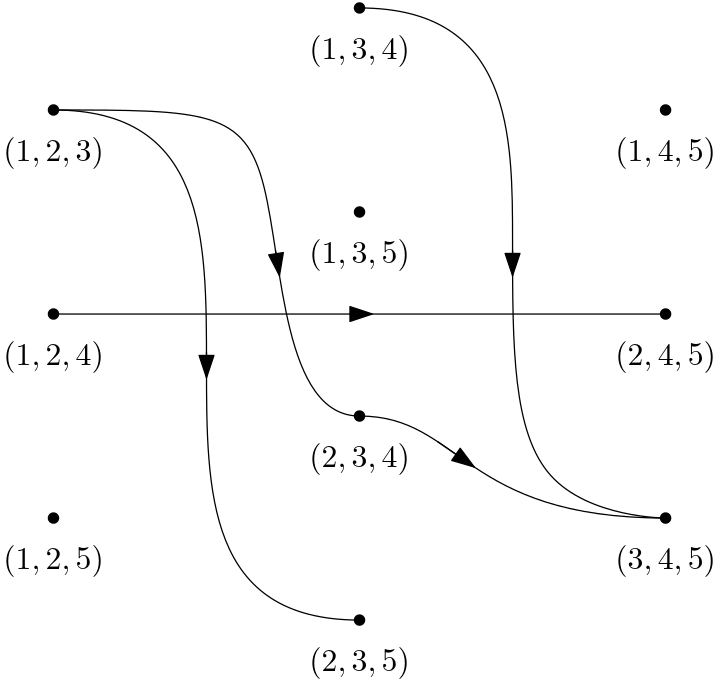}
    \caption{Directed shift graph $\vec{S}_{5,3}$.}
    \label{fig:shift}
\end{figure}

\begin{theorem}\label{thm:counterexample-flipped}
    For each integer $k \geq 3$ there exists an $\varepsilon > 0$ such that for all integers $n$, and $c$, and for each orientation $C$ of a $k$-cycle containing exactly two blocks, one of which has length one, there exists a digraph $D$ with $|V(D)| \geq n$, $\chi(D) \geq c$, and minimum out-degree at least $\varepsilon |V(D)|$, such that $D$ does not contain a subdigraph isomorphic to $C$.
\end{theorem}
\begin{proof}
    Let $m$ be sufficiently large such that $|V(\vec{S}_{2m, 2k})| \geq n$ and $\chi(\vec{S}_{2m,2k}) \geq c$. 

    We consider a supergraph $\vec{G}_{2m, 2k}$ of $\vec{S}_{2m, 2k}$ defined as follows; $V(\vec{G}_{2m,2k}) = \{(a_1, \ldots, a_{2k}) \in [2m]^{2k} \, : \, a_i \neq a_j \text{ for all } 1 \leq i < j \leq 2k\}$ and a vertex $(a_1, \ldots, a_{2k}) \in V(\vec{G}_{2m,2k})$ is out-adjacent to $(b_1, \ldots, b_{2k}) \in V(\vec{G}_{2m,2k})$ if $a_{i+1} = b_i$ for all $i \in [2k-1]$. That is, $\vec{G}_{2m, 2k}$ is defined analogously to the directed shift graph $\vec{S}_{2m, 2k}$, except that the tuples are not required to be in ascending order. Since $\vec{G}_{2m, 2k}$ contains $\vec{S}_{2m, 2k}$, we have that $|V(\vec{G}_{2m, 2k})|\geq n$ and $\chi(\vec{G}_{2m, 2k}) \geq c$.

    \begin{claim}\label{claim:firstclaim}
        $\vec{G}_{2m, 2k}$ does not contain a subdigraph isomorphic to a digraph in $\mathcal{C}_k \cup \mathcal{C}_k'$.
    \end{claim}
    \begin{subproof}[Proof of \autoref{claim:firstclaim}]
        Suppose for the sake of contradiction that $\vec{G}_{2m, 2k}$ does contain a subdigraph $H$ isomorphic to a digraph $C^*$ in $\mathcal{C}_k \cup \mathcal{C}_k'$. Let $\ell := |V(H)|$, and let $v_1, \ldots, v_{\ell}$ be an ordering of the vertices of $H$ such that $v_1 \rightarrow v_2 \rightarrow \ldots \rightarrow v_{\ell}$, and either $v_{\ell}v_1 \in E(H)$, or $v_1v_{\ell} \in E(H)$, depending on whether $C^* \in \mathcal{C}_k$ or $C^* \in \mathcal{C}_k'$. 

        For each $i \in [\ell]$ and $j \in [2k]$, let $v_{i,j}$ denote the $j^{\text{th}}$ entry in the tuple associated to $v_i$. That is, $v_i = (v_{i,1}, \ldots, v_{i, 2k})$. Let $e \in [2m]$ be the element such that $v_{1, 2k} = e$. By the definition of $\vec{G}_{2m, 2k}$, the adjacency of the vertices $v_1, \ldots, v_{\ell}$, and as $\ell < 2k$, we have that $v_{i, 2k - i + 1} = e$ for all $i \in [\ell]$. In particular, $v_{\ell, 2k - \ell + 1} = e$. Since $\ell \geq 3$, and as each element in $[2m]$ appears at most once in each tuple defining a vertex, $v_{\ell}$ and $v_1$ cannot be adjacent, contradicting the definition of $H$. By this contradiction, we conclude that the claim holds.
    \end{subproof}

    Note that by \autoref{claim:firstclaim}, $\vec{G}_{2m, 2k}$ does not contain a subdigraph isomorphic to $C$. Since \linebreak $|V(\vec{G}_{2m,2k})|\geq n$ and $\chi(\vec{G}_{2m, 2k}) \geq c$, the only remaining concern is the minimum out-degree requirement. Indeed, $\vec{G}_{2m, 2k}$ has minimum out-degree $2m - 2k + 1$, which is not a constant factor  of $|V(\vec{G}_{2m,2k})|$. To address this issue, we will add additional structures to increase the minimum out-degree.

    First, for each ordered partition $(A, B)$ of $[2m]$ with $|A| = |B| = m$, add vertices $s_{(A,B)}$ and $t_{(A,B)}$. For each vertex $s_{(A,B)}$ we add arcs from all vertices of the form $(a_1, \ldots, a_k, b_1, \ldots, b_k) \in V(\vec{G}_{2m,2k})$ where $a_1, \ldots, a_k \in A$ and $b_1, \ldots, b_k \in B$ to $s_{(A,B)}$. For each vertex $t_{(A, B)}$ we add arcs from $t_{(A,B)}$ to all vertices of the form $(a_1, \ldots, a_k, b_1, \ldots, b_k) \in V(\vec{G}_{2m,2k})$ where $a_1, \ldots, a_k \in A$ and $b_1, \ldots, b_k \in B$. Let $S$ and $T$ be the sets of all vertices of the form $s_{(A,B)}$ and of the form $t_{(A,B)}$ respectively. Next, we add a directed path $P$ of length $k - 1$ given by vertices $p_1, \ldots, p_k$ where $p_ip_{i+1} \in E(P)$ for all $i \in [k-1]$. Finally, make $S$ out-complete to $p_1$, and $T$ in-complete from $p_k$.

    Let $D'$ be the resulting digraph. Thus, $V(D') = V(\vec{G}_{2m, 2k}) \cup S \cup T \cup V(P)$.

    \begin{claim}\label{claim:secondclaim}
        $D'$ does not contain a subdigraph isomorphic to a digraph in $\mathcal{C}_k \cup \mathcal{C}_k'$.
    \end{claim}
    \begin{subproof}[Proof of \autoref{claim:secondclaim}]
        Suppose for the sake of contradiction that $D'$ does contain a subdigraph $H$ isomorphic to a digraph $C^*$ in $\mathcal{C}_k \cup \mathcal{C}_k'$. We first note that due to path $P$ having length $k$ and only the endpoints being connected with vertices in $V(D')\setminus V(P)$, the internal vertices of path $P$ do not lie in any orientation of a $k$-cycle in $D'$. Thus, $p_2, \ldots, p_{k-1} \not\in V(H)$. Suppose that $p_1 \in V(H)$. Since the sole out-neighbour of $p_1$ in $D'$ is $p_2$, it follows that $p_1$ is contained in two blocks of $H$, and thus $C^* \in \mathcal{C}_k'$. Let $u \in V(H)$ be the vertex such that $H[\{u, p_1\}]$ is the block of length one in $H$. As all in-neighbours of $p_1$ in $D$ are in $S$, it follows that $u \in S$. Let $v \in V(H)$ be the next vertex along the cycle after $u$. Because $p_1$ is the only out-neighbour of vertices in $S$ in $D'$, it follows that $H[\{v, u, p_1\}]$ forms a directed subpath of $H$. However, this contradicts $H[\{u, p_1\}]$ being a block of $H$. By this contradiction, we conclude that $p_1 \not \in V(H)$. Analogously it follows that $p_k \not \in V(H)$. Hence, $V(H) \cap V(P) = \emptyset$.

        Now suppose that $H$ contains a vertex $s_{(A,B)} \in S$. As $s_{(A,B)}$ has no out-neighbours in $D' - V(P)$, the vertex $s_{(A,B)}$ must be contained in two blocks of $H$. Therefore, we have that $C^* \in \mathcal{C}_k'$. Let $u = (a_1, \ldots, a_k, b_1, \ldots, b_k)$ and $u' = (a_1', \ldots, a_k', b_1', \ldots, b_k')$ be the two vertices adjacent to $s_{(A,B)}$ in $H$, where $a_1, \ldots, a_k, a_1', \ldots, a_k' \in A$ and $b_1, \ldots, b_k, b_1', \ldots, b_k' \in B$. By the structure of $C^*$, it follows that $H$ contains a directed path of length $k-2$ between $u$ and $u'$. Without loss of generality, assume that this path goes from $u$ to $u'$. We note that this path does not include vertices in $S$ or $T$, as those have no out- and in-neighbours in $D' - V(P)$ respectively. Then, by the definition of the arcset of $\vec{G}_{2m,2k}$, it follows that $a_3' = b_k$. Thus, $a_3' \in A$ and $a_3' \in B$, a contradiction. Thus, $V(H) \cap S = \emptyset$, and analogously it follows that $V(H) \cap T = \emptyset$.

        Thus, $H$ must be a subdigraph of $D' - (V(P) \cup S \cup T) = \vec{G}_{2m,2k}$. This contradicts \autoref{claim:firstclaim}.
    \end{subproof}

    The digraph $D'$ still does not have the desired minimum out-degree. Each vertex in $V(\vec{G}_{2m,2k})$ has out-degree $2m - 2k + 1 + \binom{2m-2k}{m-k}$ in $D'$. Vertices in $S$ and $V(P) \setminus \{p_k\}$ have out-degree one, $p_k$ has out-degree $\binom{2m}{m}$, and all vertices in $T$ have out-degree $\binom{m}{k}^2\cdot (k!)^2$. To achieve the desired minimum out-degree, we will clone vertices and use \autoref{lem:cloning} to show that the resulting digraph does not contain an orientation of a $k$-cycle consisting of two blocks, one of which has length one.

    Consider the sets $V(\vec{G}_{2m, 2k})$, $S$, $T$, and $V(P)$. We aim for these sets to have approximately the same size. In $D'$, the sizes of these sets are $\frac{(2m)!}{(2m - 2k)!}$, $\binom{2m}{m}$, $\binom{2m}{m}$, and $k$ respectively. Let $q = \max(\frac{(2m)!}{(2m - 2k)!}, \binom{2m}{m}, k)$. For each of the sets $V(\vec{G}_{2m, 2k})$, $S$, $T$, and $V(P)$, repeatedly double the size of the set by cloning each vertex in the set once until the set reaches a size of at least $q/2$. Let $D$ be the resulting digraph. See \autoref{fig:construction} for an illustration of digraph $D$.

    Note that in $D$, each of the sets $V(\vec{G}_{2m, 2k})$, $S$, $T$, and $V(P)$ has size at least $V(D) / 8$. Moreover, we observe that each vertex in $S$ and $V(P) \setminus \{p_k\}$ in $D'$ is out-adjacent to a $1/k$ fraction of the vertices in $V(P)$ and vertex $p_k$ is out-adjacent to all of $T$. Next, we claim that each vertex in $T$ is out-adjacent to at least a $1/2^{2k}$ fraction of $V(\vec{G}_{2m,2k})$ and each vertex in $V(\vec{G}_{2m,2k})$ is out-adjacent to at least a $1/2^{2k}$ fraction of $S$. Namely, for a vertex $t_{A,B} \in T$, consider all its out-neighbours of the form $(a_1, \ldots, a_k, b_1, \ldots, b_k)$ with $a_1, \ldots, a_k \in A$ and $b_1, \ldots, b_k \in B$ by choosing the elements in the order $a_1, b_1, a_2, b_2, \ldots, a_k, b_k$. When choosing $a_i$ for $i \in [k]$, there are $m-i + 1$ remaining choices for elements in $A$, and $2m - 2i +2$ remaining elements in total. Similarly, when choosing $b_i$ for $i \in [k]$, there are $m-i + 1$ remaining choices for elements in $B$, and $2m-2i + 1$ remaining elements in total. Hence, for each of the $2k$ choices, at least half of the options are valid choices. Thus, each vertex in $T$ is out-adjacent to at least a $1/2^{2k}$ fraction of $V(\vec{G}_{2m,2k})$. Since the underlying undirected bipartite graph between $\vec{G}_{2m,2k}$ and $S$ is biregular, by an analogous argument and using symmetry, each vertex in $V(\vec{G}_{2m,2k})$ is out-adjacent to at least a $1/2^{2k}$ fraction of $S$. These lower bounds on the relative out-degrees are unaffected by the repeated cloning, as within each of the parts $V(\vec{G}_{2m,2k}), S, T, V(P)$ each vertex is cloned an equal number of times. Hence, these relations also hold for $D$. 

    Then, as each of the sets $V(\vec{G}_{2m, 2k})$, $S$, $T$, and $V(P)$ has size at least $V(D) / 8$, each vertex in $D$ has minimum out-degree at least $1/2^{2k} \cdot 1/8 = 2^{-2k-3}$. Hence, for $\varepsilon \leq 2^{-2k-3}$, digraph $D$ has minimum out-degree at least $\varepsilon|V(D)|$, as desired. 
    
    Moreover, by \autoref{lem:cloning} and \autoref{claim:secondclaim}, $D$ does not contain a subdigraph isomorphic to a digraph in $\mathcal{C}_k \cup \mathcal{C}_k'$. Thus, $D$ does not contain an orientation of a $k$-cycle consisting of two blocks, one of which has length one, as desired.
\end{proof}

\begin{figure}
    \centering
    \includegraphics[width=0.8\linewidth]{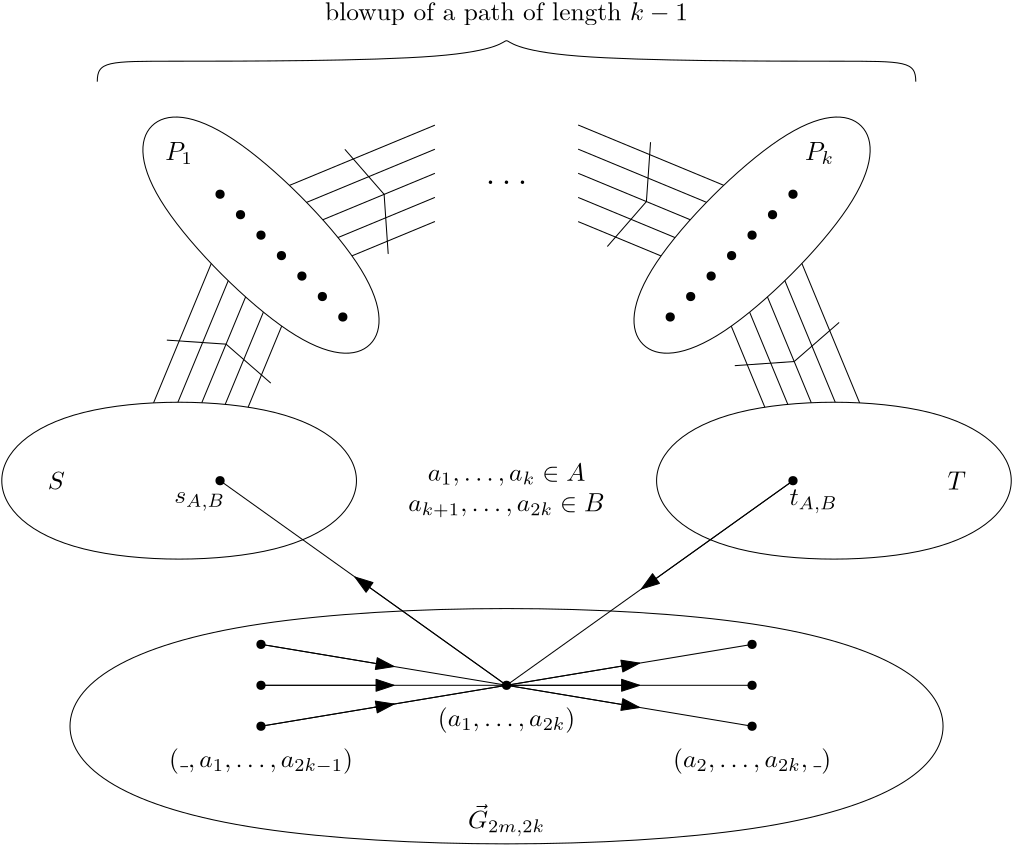}
    \caption{Illustration of the construction used in the proof of \autoref{thm:counterexample-flipped}.}
    \label{fig:construction}
\end{figure}

We note that the construction in the proof of \autoref{thm:counterexample-flipped} additionally does not contain directed $k$-cycles, and thus provides an alternate proof for \autoref{thm:counterexample-directed}.

Observing that directed $2$-cycles and all orientations of triangles are covered by \autoref{thm:counterexample-directed} and \autoref{thm:counterexample-flipped}, noting that the other orientation of $2$-cycles does not occur due to the exclusion of parallel arcs, and using that $|V(D)| \geq \chi(D)$ for all digraphs $D$, \autoref{thm:maintheorem} now follows directly from \autoref{lem:threeblocks}, \autoref{lem:twoblockseachlengthtwo}, \autoref{thm:counterexample-directed}, and \autoref{thm:counterexample-flipped}.

\section{Related problems}
The results in this paper relate to the following question, which is a more general version of \autoref{question:mainquestion}:
\begin{question}\label{question:further_question1}
    For which fixed digraphs $H$ does there exist a constant $c_{\varepsilon} \geq 1$ for all $\varepsilon > 0$ such that every digraph $D$ without loops or parallel arcs with $\chi(D) \geq c_{\varepsilon}$ and minimum out-degree at least $\varepsilon |V(D)|$ contains $H$ as a subdigraph?
\end{question}
Specifically, \autoref{thm:maintheorem} characterizes the orientations of cycles for which \autoref{question:further_question1} is answered in the affirmative. 

Combining the fact that each sufficiently large graph with a linear number of edges contains a large complete bipartite subgraph~\cite{alon2003turan} with a Ramsey-type argument where we colour the arcs based on their orientation between the two sets of the bipartition, we obtain that each sufficiently large digraph with a linear number of arcs contains a large directed complete bipartite subdigraph where all the arcs have the same orientation between the two sets of the bipartition. Hence, \autoref{question:further_question1} has an affirmative answer for all subdigraphs of such directed complete bipartite digraphs.

For digraphs for which \autoref{question:further_question1} is not answered in the affirmative, it may be interesting to ask for which values of $\varepsilon$ the statement of \autoref{question:further_question1} does hold.

\begin{question}\label{question:directed_chromatic_threshold}
    For a fixed digraph $H$, what is the minimum value $\varepsilon_H \geq 0$ such that there exists a constant $c \geq 1$ such that every digraph $D$ without loops or parallel arcs with $\chi(D) \geq c$ and minimum out-degree at least $\varepsilon_H |V(D)|$ contains $H$ as a subdigraph?
\end{question}
Note that \autoref{question:directed_chromatic_threshold} is well-defined for all digraphs $H$, as the statement trivially holds for $\varepsilon_H = 1$. The notion $\varepsilon_H$ in \autoref{question:directed_chromatic_threshold} provides a directed analogue to the chromatic threshold as studied for undirected graphs. \autoref{thm:maintheorem} may be interpreted as characterizing the orientations of cycles for which the directed chromatic threshold is equal to zero. 

It would additionally be interesting to study variants of \autoref{question:further_question1} and \autoref{question:directed_chromatic_threshold} where instead of considering digraphs of large chromatic number, one considers digraphs of large dichromatic number, a directed analogue of the chromatic number introduced by Neumann-Lara~\cite{NEUMANNLARA1982265}.

\bibliographystyle{plain}
\bibliography{thebib.bib}

\end{document}